\documentclass[12pt]{amsart}

\usepackage{amsthm}
\usepackage{amssymb}
\usepackage{amsmath}

\numberwithin{equation}{section}
\newtheorem{thm}{Theorem}[section]

\newtheorem{lem}[thm]{Lemma}

\theoremstyle{definition}

\newtheorem{rem}[thm]{Remark}

\newcommand{\bk}{\mathbf{k}}

\begin{document}

\title[On a weighted sum of multiple $T$-values of fixed weight and depth]
{On a weighted sum of multiple $T$-values \\ of fixed weight and depth}
\author{Yoshihiro Takeyama}
\address{Department of Mathematics,
Faculty of Pure and Applied Sciences,
University of Tsukuba, Tsukuba, Ibaraki 305-8571, Japan}
\email{takeyama@math.tsukuba.ac.jp}

\begin{abstract}
The multiple $T$-value, which is a variant of multiple zeta value of level two,
is introduced by Kaneko and Tsumura.
We show that the generating function of a weighted sum of the multiple $T$-values
of fixed weight and depth is given in terms of the multiple $T$-values of depth one
by solving a differential equation of Heun type.
\end{abstract}
\maketitle

\setcounter{section}{0}
\setcounter{equation}{0}


\section{Introduction}

We call a tuple of positive integers an \textit{index}.
The \textit{weight}, \textit{depth} and \textit{height} of an index
$\bk=(k_{1}, \ldots , k_{n})$ are defined by
$k=\sum_{j=1}^{n}k_{j}, \, n$ and $s=\left|\{ j \mid k_{j}\ge 2 \}\right|$,
respectively.
An index $\bk=(k_{1}, \ldots , k_{n})$ is called \textit{admissible} if $k_{n} \ge 2$.
Denote by $I(k, n, s)$ the set of indices of weight $k$, depth $n$ and height $s$,
and by $I_{0}(k, n, s)$ the subset of $I(k, n, s)$ consisting of admissible indices.

In this paper we consider a generating function of
the \textit{multiple $T$-value} introduced by Kaneko and Tsumura \cite{KT1}
as a variant of multiple zeta value of level two.
The multiple $T$-value is defined by
\begin{align*}
T(\bk)=2^{n}\sum_{\substack{0<m_{1}<\cdots <m_{n} \\
\forall{i}: \, m_{i}\equiv i \, \hbox{\scriptsize mod}\, 2}}
\frac{1}{m_{1}^{k_{1}}\cdots m_{n}^{k_{n}}}
\end{align*}
for an admissible index $\bk=(k_{1}, \ldots , k_{n})$.
Our result is the following formula for the generating function of
the weighted sum of the multiple $T$-values of fixed weight and depth:

\begin{thm}\label{thm:main}
It holds that
\begin{align}
&
1-\sum_{k>n \ge 1}
\left(\sum_{s=1}^{n}\frac{1}{2^{s}}\sum_{\bk \in I_{0}(k, n, s)}T(\bk) \right)x^{k-n}y^{n}
\label{eq:main} \\
&\quad {}=\exp{\left(
\sum_{n\ge 2}\frac{T(n)}{2n}(x^{n}+y^{n}-(x+y)^{n})
\right)}.
\nonumber
\end{align}
\end{thm}

It is known that the $\mathbb{Q}$-vector space spanned by the multiple $T$-values is
closed with respect to multiplication.
Hence we obtain linear relations among the multiple $T$-values from Theorem \ref{thm:main} in principle
by expanding the products of $T(n)$'s in the Taylor expansion of the right side of \eqref{eq:main}.
For example, by calculating the coefficients of $x^{k-2}y^{2}$ for $k \ge 3$, we find that
\begin{align*}
{}-\left(\frac{1}{2}\,T(1,k-1)+\frac{1}{4}\sum_{j=2}^{k-2}T(k-j,j)\right)=
{}-\frac{k-1}{4}T(k)+\frac{1}{8}\sum_{j=2}^{k-2}T(j)T(k-j).
\end{align*}
Combining it with the shuffle product relation (see Eqn.(4.8) in \cite{KT2})
\begin{align*}
\frac{1}{2}\sum_{j=2}^{k-2}T(j)T(k-j)=
(2^{k-2}-2)T(1, k-1)+\sum_{j=2}^{k-2}(2^{j-1}-1)T(k-j, j),
\end{align*}
we reproduce the weighted sum formula
\begin{align*}
\sum_{j=2}^{k-1}2^{j-1}T(k-j, j)=(k-1)T(k)
\end{align*}
due to Kaneko and Tsumura \cite[Theorem 3.2]{KT2}.
See \cite{BCJXXZ} and the references therein for recent results on weighted sum
of variants of multiple zeta values.

The proof of \eqref{eq:main} is quite similar to
that by Ohno and Zagier \cite{OZ}
for a formula for the generating function of the sum of
the multiple zeta values
\begin{align*}
\zeta(\bk)=\sum_{\substack{0<m_{1}<\cdots <m_{n}}}
\frac{1}{m_{1}^{k_{1}}\cdots m_{n}^{k_{n}}}
\end{align*}
of fixed weight, depth and height.
They consider the function
\begin{align*}
\Phi_{0}(x, y, z; t)=\sum_{k>n\ge 1, \, s\ge 1}
\left(\sum_{\bk \in I_{0}(k, n, s)} L_{\bk}(t) \right)
x^{k-n-s}y^{n-s}z^{s-1},
\end{align*}
where $L_{\bk}(t)$ is the multiple polylogarithm of one variable
\begin{align*}
L_{\bk}(t)=\sum_{\substack{0<m_{1}<\cdots <m_{n}}}
\frac{t^{m_{n}}}{m_{1}^{k_{1}}\cdots m_{n}^{k_{n}}},
\end{align*}
and show that it is the unique solution of the differential equation
\begin{align*}
t(1-t)\frac{d^{2}\Phi_{0}}{dt^{2}}+\left((1-x)(1-y)-yt\right)\frac{d\Phi_{0}}{dt}+(xy-z)\Phi_{0}=1
\end{align*}
which is holomorphic in a neighborhood of $t=0$ and satisfies $\Phi_{0}(0)=0$.
The solution is written in terms of the hypergeometric function
$F(\alpha, \beta, \gamma; t)$.
Using the Gauss summation formula
\begin{align*}
F(\alpha, \beta, \gamma; 1)=\frac{\Gamma(\gamma)\Gamma(\gamma-\alpha-\beta)}
{\Gamma(\gamma-\alpha) \Gamma(\gamma-\beta)},
\end{align*}
we see that the generating function $\Phi_{0}(x, y, z; 1)$ of the multiple zeta values
is written in terms of Riemann zeta values $\zeta(n)$ with $n \ge 2$.

We will carry out the same calculation for the multiple $T$-values.
In this case
we have to solve Heun's equation, which is a Fuchsian equation of second order
with four singularities.
Although no closed explicit formula for solutions of Heun's general equation is known,
there are several methods to find solutions under special conditions of its parameters.
For our generating function in Theorem \ref{thm:main},
the technique found by Ishkhanyan and Suominen \cite{IS} works,
and we obtain the formula \eqref{eq:main} for the multiple $T$-values
in the same way as for the multiple zeta values.


\section{Proof of Theorem \ref{thm:main}}

In \cite{S} Sasaki defines the polylogarithm of level two.
In the proof of Theorem \ref{thm:main}
we use its multiple version introduced by Kaneko and Tsumura \cite{KT1}.
For an index $\bk=(k_{1}, \ldots , k_{n})$,
the multiple version of the polylogarithm of level two is defined by
\begin{align*}
\mathrm{Ath}(\bk; t)=
\sum_{\substack{0<m_{1}<\cdots <m_{n} \\
\forall{i}: \, m_{i}\equiv i \,  \hbox{\scriptsize mod}\, 2}}
\frac{t^{m_{n}}}{m_{1}^{k_{1}}\cdots m_{n}^{k_{n}}}.
\end{align*}
For the empty index $\bk=\emptyset$ we set $\mathrm{Ath}(\emptyset; t)=1$.
The function $\mathrm{Ath}(\bk; t)$ is holomorphic in the open disk $|t|<1$,
and, if $\bk$ is admissible, continuous on the closed disk $|t|\le 1$.
We have
\begin{align}
\frac{d}{dt}\mathrm{Ath}(\bk; t)=\left\{
\begin{array}{ll}
\displaystyle
\frac{1}{t}\, \mathrm{Ath}(k_{1}, \ldots , k_{n-1}, k_{n}-1; t) & (k_{n} \ge 2) \\
\displaystyle
\frac{1}{1-t^{2}}\, \mathrm{Ath}(k_{1}, \ldots , k_{n-1}; t) & (k_{n}=1)
\end{array}
\right.
\label{eq:ath-recursion}
\end{align}
for any non-empty index $\bk=(k_{1}, \ldots, k_{n})$.

For $n \ge 1$ and $k, s \ge 0$ we set
\begin{align*}
G(k, n, s; t)=2^{n}\sum_{\bk \in I(k, n, s)}\mathrm{Ath}(\bk; t), \quad
G_{0}(k, n, s; t)=2^{n}\sum_{\bk \in I_{0}(k, n, s)}\mathrm{Ath}(\bk; t).
\end{align*}
If the range of the sum is empty, the right side is zero.
Note that
\begin{align}
G_{0}(k, n, s; 1)=\sum_{\bk \in I_{0}(k, n, s)}T(\bk).
\label{eq:G=T}
\end{align}
{}From \eqref{eq:ath-recursion} we see that
\begin{align}
&
\frac{d}{dt}G_{0}(k, n, s; t)
\label{eq:diffG-1} \\
&\quad {}=\frac{1}{t} \left(
G_{0}(k-1, n, s; t)+G(k-1, n, s-1; t)-G_{0}(k-1, n, s-1; t)
\right),
\nonumber \\
&
\frac{d}{dt}\left( G(k, n, s; t)-G_{0}(k, n, s; t) \right)=\frac{2}{1-t^{2}}G(k-1, n-1, s; t),
\label{eq:diffG-2}
\end{align}
where we set
\begin{align*}
G(k, 0, s; t)=\left\{
\begin{array}{ll}
1 & ((k, s)=(0,0)), \\ 0 & \hbox{(otherwise)}.
\end{array}
\right.
\end{align*}

Now we consider the generating functions
\begin{align}
&
\Phi(t)=\Phi(x, y, z; t)=1+\sum_{\substack{k, n \ge 1 \\ s \ge 0}}
G(k, n, s; t)x^{k-n-s}y^{n-s}z^{2s},
\label{eq:def-Phi} \\
&
\Phi_{0}(t)=\Phi_{0}(x, y, z; t)=\sum_{\substack{n, k, s \ge 1}}
G_{0}(k, n, s; t)x^{k-n-s}y^{n-s}z^{2(s-1)}.
\label{eq:def-Phi0}
\end{align}
Although we will set $z^{2}=xy/2$ to derive \eqref{eq:main},
we keep the parameter $z$ for a while.

\begin{lem}\label{lem:unif-conv1}
Suppose that $|x|<1/2, |y|<1/4$ and $|z^{2}|<|xy|$.
\begin{enumerate}
  \item The right side of \eqref{eq:def-Phi} converges uniformly on the closed disk $|t|\le 1/2$.
  \item The right side of \eqref{eq:def-Phi0} converges uniformly on the closed disk $|t|\le 1$.
\end{enumerate}
\end{lem}

\begin{proof}
Assume that $|t|\le 1/2$.
Since $m \le 2^{m}$ for $m \ge 1$, we have
\begin{align*}
\left|\mathrm{Ath}(\bk; t)\right| \le T(k_{1}, \ldots , k_{n-1}, k_{n}+1) \le
T(\underbrace{1, \ldots , 1}_{n-1}, 2)
\end{align*}
for any index $\bk=(k_{1}, \ldots , k_{n})$ of depth $n$.
Using the relation
\begin{align*}
T(\underbrace{1, \ldots , 1}_{n-1}, 2)=T(n+1),
\end{align*}
which is a special case of the duality of the multiple $T$-values proved in \cite{KT2},
and the inequality $T(n+1)\le 2$, we obtain
\begin{align*}
\left|G(k, n, s; t) \right| \le 2^{n+1} \left|I(k, n, s) \right| \le
2^{n+1}\sum_{s=0}^{n}\left|I(k, n, s) \right|=2^{n+1}\binom{k-1}{n-1}
\end{align*}
for $n \ge 1$.
Under the assumption on $x, y$ and $z$, it holds that
\begin{align*}
\sum_{\substack{k, n \ge 1 \\ s \ge 0}} 2^{n+1}\binom{k-1}{n-1}
\left|x^{k-n-s}y^{n-s}z^{2s} \right|=
\frac{4|y|}{(1-|z^{2}/xy|)(1-|x|-2|y|)}<+\infty.
\end{align*}
Therefore $\Phi(x, y, z; t)$ converges uniformly on $|t|\le 1/2$,
and this completes the proof of (i).
If $\bk$ is admissble, we have $|\mathrm{Ath}(\bk; t)| \le 2$ on
the closed disk $|t|\le 1$, which implies (ii).
\end{proof}

Because of Lemma \ref{lem:unif-conv1},
term-by-term differentiation is allowed and we see that
\begin{align*}
&
\frac{d}{dt}\Phi_{0}(t)=\frac{x}{t}\Phi_{0}(t)+\frac{1}{yt}\left(\Phi(t)-1-z^{2}\Phi_{0}(t)\right), \\
&
\frac{d}{dt}\left(\Phi(t)-z^{2}\Phi_{0}(t)\right)=\frac{2y}{1-t^{2}}\Phi(t)
\end{align*}
by using \eqref{eq:diffG-1} and \eqref{eq:diffG-2}.
Eliminating $\Phi(t)$ we obtain
\begin{align*}
t(1-t^2)\Phi_{0}''(t)+\left\{(1-x)(1-t^2)-2ty\right\}\Phi_{0}'(t)+2(xy-z^2)\Phi_{0}(t)=2.
\end{align*}
Set
\begin{align}
u(t)=1-(xy-z^{2})\Phi_{0}(t).
\label{eq:phi-to-u}
\end{align}
Then $u(t)$ is the unique solution of the homogeneous equation
\begin{align}
t(1-t^2)u''+\left\{(1-x)(1-t^2)-2ty\right\}u'+2(xy-z^2)u=0
\label{eq:u}
\end{align}
which is regular in the neighborhood of $t=0$ and satisfies $u(0)=1$.

The equation \eqref{eq:u} is Heun's equation with singularities at
$t=0, 1, \infty$ and $-1$.
In such a case, Heun's equation can be solved if
the parameters satisfy a special condition as discussed in \cite{IS}.
We change the dependent variable $u$ to
\begin{align*}
v(t)=t^{1+x}(1-t)^{y}(t^{-x}u(t))',
\end{align*}
where $t^{x}$ and $(1-t)^{y}$ are the principal values.
Then $v(t)$ is the unique solution of the equation
\begin{align}
t(1-t^2)v''-\left\{x(1-t^2)-2(y-1)t^{2}\right\}v'-\left\{(2z^{2}-xy)+y(y+x-1)t\right\}v=0
\label{eq:v}
\end{align}
which is holomorphic in a neighborhood of $t=0$ and satisfies $v(0)=-x$.
Now we consider the case where
\begin{align*}
z^{2}=xy/2.
\end{align*}
Then the solution $v(t)$ is given by
\begin{align*}
v(t)=-x\, F(-\frac{y}{2}, \frac{1-x-y}{2}, \frac{1-x}{2}; t^{2}),
\end{align*}
where $F(\alpha, \beta, \gamma; t)$ is the hypergeometric function
\begin{align*}
F(\alpha, \beta, \gamma; t)=\sum_{n \ge 0}
\frac{(\alpha)_{n}(\beta)_{n}}{n! \, (\gamma)_{n}}t^{n}.
\end{align*}
Here $(a)_{n}$ is the shifted factorial
$(a)_{n}=\prod_{j=0}^{n-1}(a+j)$.

For the time being we assume that
\begin{align}
\mathrm{Re}{x}<0, \qquad 0<\mathrm{Re}{y}<1.
\label{eq:assume-xy}
\end{align}
Then it holds that
\begin{align}
u(t)&=-x \, t^{x}\int_{0}^{t}ds \, s^{-1-x}(1-s)^{-y}
 F(-\frac{y}{2}, \frac{1-x-y}{2}, \frac{1-x}{2}; s^{2})
 \label{eq:formula-u} \\
 &=-x \, t^{x}\sum_{m \ge 0}
\frac{(-y/2)_{m}((1-x-y)/2)_{m}}{m! ((1-x)/2)_{m}}
 \int_{0}^{t}ds \, s^{2m-1-x}(1-s)^{-y}
 \nonumber \\
 &=-x\sum_{m \ge 0} t^{2m}
\frac{(-y/2)_{m}((1-x-y)/2)_{m}}{m! ((1-x)/2)_{m}}
 \int_{0}^{1}ds \, s^{2m-1-x}(1-ts)^{-y}
\nonumber
\end{align}
in the open disk $|t|<1$.
Thus we obtain the formula for the holomorphic solution $u(t)$ of
the differential equation \eqref{eq:u} of Heun type in the case of $z^{2}=xy/2$.

To prove Theorem \ref{thm:main} we take the limit of \eqref{eq:phi-to-u} as $t \to 1-0$.
{}From \eqref{eq:G=T} and Lemma \ref{lem:unif-conv1} (ii),
we see that the right side of \eqref{eq:phi-to-u} with $z^{2}=xy/2$
converges to the left side of \eqref{eq:main}.
Let us calculate the limit of $u(t)$.

\begin{lem}
Under the assumption \eqref{eq:assume-xy},
the right side of \eqref{eq:formula-u} converges uniformly
on the interval $0\le t \le 1$.
\end{lem}

\begin{proof}
For $0\le t \le 1$ and $m \ge 0$, we have
\begin{align*}
&
\left|  t^{2m} \frac{(-y/2)_{m}((1-x-y)/2)_{m}}{m! ((1-x)/2)_{m}}
 \int_{0}^{1}ds \, s^{2m-1-x}(1-ts)^{-y} \right|  \\
 &\le
 \left| \frac{(-y/2)_{m}((1-x-y)/2)_{m}}{m! ((1-x)/2)_{m}} \right|
\int_{0}^{1}ds\, s^{-1-\mathrm{Re}{x}}(1-s)^{-\mathrm{Re}{y}}
\end{align*}
under the assumption \eqref{eq:assume-xy}.
Since the series $F(\alpha, \beta, \gamma; 1)$ converges absolutely if
$\mathrm{Re}(\alpha+\beta-\gamma)<0$,
the infinite sum
\begin{align*}
\sum_{m \ge 0}
 \left| \frac{(-y/2)_{m}((1-x-y)/2)_{m}}{m! ((1-x)/2)_{m}} \right|
\end{align*}
converges if $y$ satisfies \eqref{eq:assume-xy}.
Hence the right side of \eqref{eq:formula-u} passes the Weierstrass $M$-test on $0\le t \le 1$.
\end{proof}

Thus we obtain
\begin{align}
\lim_{t \to 1-0}u(t)&=-x \sum_{m \ge 0}
\frac{(-y/2)_{m}((1-x-y)/2)_{m}}{m! ((1-x)/2)_{m}}
\frac{\Gamma(2m-x)\Gamma(1-y)}{\Gamma(2m+1-x-y)}
\label{eq:limit-u} \\
&=\frac{\Gamma(1-x)\Gamma(1-y)}{\Gamma(1-x-y)}
\sum_{m \ge 0}
\frac{(-y/2)_{m}(-x/2)_{m}}{m! (1-(x+y)/2)_{m}}
\nonumber \\
&=\frac{\Gamma(1-x)\Gamma(1-y)}{\Gamma(1-x-y)}
F(-\frac{y}{2}, -\frac{x}{2}, 1-\frac{x+y}{2}; 1)
\nonumber \\
&=\frac{\Gamma(1-x)\Gamma(1-y)}{\Gamma(1-x-y)}
\frac{\Gamma(1-(x+y)/2)}{\Gamma(1-x/2)\Gamma(1-y/2)}.
\nonumber
\end{align}
Using
\begin{align*}
\Gamma(1-t)=\exp{\left(\gamma t+\sum_{n \ge 2}\frac{\zeta(n)}{n}t^{n}\right)},
\end{align*}
where $\gamma$ is Euler's constant, and
\begin{align*}
T(n)=2\left(1-\frac{1}{2^{n}}\right)\zeta(n),
\end{align*}
we see that the right side of \eqref{eq:limit-u} is equal to that of \eqref{eq:main},
which is holomorphic in a neighborhood of $(x, y)=(0, 0)$.
Therefore we can omit the assumption \eqref{eq:assume-xy}.
This completes the proof of Theorem \ref{thm:main}.

\begin{rem}
The equation \eqref{eq:v} can be solved explicitly also in the case of $z=0$,
and the solution is $v(t)=(1+t)^{y}$.
{}From it we reproduce the formula for the generating function of height one $T$-values
\begin{align*}
&
1-\sum_{k>n \ge 1}T(\underbrace{1, \ldots , 1}_{n-1}, k-n+1)x^{k-n}y^{n} \\
&=\frac{2 \Gamma(1-x)\Gamma(1-y)}{\Gamma(1-x-y)}
F(1-x, 1-y, 1-x-y; -1)
\end{align*}
proved by Kaneko and Tsumura \cite{KT2}.
We omit the details of the calculation.
\end{rem}

\section*{Acknowledgement}
The author is very grateful to Yasuo Ohno for a discussion on the problem
dealt with in this article and  for encouragement.
This work is partially supported by JSPS KAKENHI Grant Number 18K03233.



\begin{thebibliography}{000}

\bibitem{BCJXXZ}
Berger, S., Chandra, A., Jain, J., Xu, D., Xu, C. and Zhao, J.,
Weighted sums of Euler sums and other variants of multiple zeta values,
preprint, arXiv:2011.02393.

\bibitem{IS}
Ishkhanyan, A. and Suominen, K. A.,
New solutions of Heun's general equation,
\textit{J. Phys. A} \textbf{36} (2003), no. 5, L81--L85.

\bibitem{KT1}
Kaneko, M. and Tsumura, H.,
Zeta functions connecting multiple zeta values and poly-Bernoulli numbers,
\textit{Adv. Stud. Pure Math. Various Aspects of Multiple Zeta Functions} (2020), 181--204.

\bibitem{KT2}
Kaneko, M. and Tsumura, H.,
On a variant of multiple zeta values of level two,
preprint, arXiv:1903.03747.

\bibitem{OZ}
Ohno, Y. and Zagier, D.,
Multiple zeta values of fixed weight, depth, and height,
\textit{Indag. Math. (N.S.)} \textbf{12} (2001), no. 4, 483--487.

\bibitem{S}
Sasaki, Y.,
On generalized poly-Bernoulli numbers and related $L$-functions,
\textit{J. Number Theory} \textbf{132} (2012), no. 1, 156--170.

\end{thebibliography}
\end{document}